\documentclass{amsart}

\usepackage{latexsym}
\usepackage{verbatim}
\usepackage{amsmath}
\usepackage{amssymb}
\usepackage{mathrsfs}
\usepackage{amsthm}
\usepackage{tikz}
\usetikzlibrary{decorations.pathreplacing}
\usetikzlibrary{arrows.meta}

\newcommand{\Ind}{
 \setbox0=\hbox{$x$}\kern\wd0\hbox to 0pt{\hss$
 \mid$\hss}\lower.9\ht0\hbox to 0pt{\hss$\smile$\hss}\kern\wd0
}

\newcommand{\Notind}{
 \setbox0=\hbox{$x$}\kern\wd0\hbox to 0pt{\mathchardef
 \nn=12854\hss$\nn$\kern1.4\wd0\hss}\hbox to 0pt{\hss$\mid$\hss}\lower.9\ht0
 \hbox to 0pt{\hss$\smile$\hss}\kern\wd0
}

\usepackage{amsthm, amsfonts}
\usepackage{xcolor}
\usepackage{euscript}
\usepackage{graphicx}
\usepackage{float}
\usepackage{caption}
\usepackage{tikz-qtree}
\usepackage{todonotes}
\usepackage{subcaption}
\renewcommand{\d}{\delta}
\newcommand{\inv}{^{-1}}
\newcommand{\C}{\mathcal{C}}
\newcommand{\K}{\mathcal{K}}

\newcommand{\cl}{\mathrm{cl}}

\newcommand{\acl}{\mathrm{acl}}

\newcommand{\dcl}{\mathrm{dcl}}

\newcommand{\tp}{\mathrm{tp}}

\newcommand{\of}{\mathcal{OF}}
\newcommand{\af}{\mathcal{AF}}

\newcommand{\expa}{extension\ }
\newcommand{\Exp}{\mathrm{Ext}}
\newcommand{\x}{x}%{\langle x\rangle}
\newcommand{\y}{y}%{\langle y\rangle}
\renewcommand{\a}{a}%{\langle a\rangle}
\renewcommand{\b}{b}%{\langle b\rangle}
\newcommand{\para}{\text{par}}
\newcommand{\ort}{\text{orth}}

\newcommand{\tile}{\text{tile}}
\newcommand{\F}{\mathcal{F}}
\newcommand{\Z}{\mathbb{Z}}
\newcommand{\N}{\mathbb{N}}
\newtheorem{defi}{Definition}[section]
\newtheorem{theorem}[defi]{Theorem}
\newtheorem{definition}[defi]{Definition}
\newtheorem{lemma}[defi]{Lemma}
\newtheorem{proposition}[defi]{Proposition}

\newtheorem{corollary}[defi]{Corollary}

\newtheorem{remark}[defi]{Remark}

\newtheorem{notation}[defi]{Notation}

\newtheorem{fact}[defi]{Fact}
\def\Aut{\mathop{\rm Aut}\nolimits}

\def\nil2{\mathop{\rm nil2}\nolimits}
\def\Th{\mathop{\rm Th}\nolimits}

\newcommand{\T}{\Th(\af_2)}

\begin{document}
\author{Zahra Mohammadi Khangheshlaghi \& Katrin Tent}
\date{\today}
\title{On the model theory of the Free Factor Complex of rank 2}
\begin{abstract}
We begin the investigation of the free factor complex of a free group of finite rank. For the case of rank $2$  we axiomatize its theory and show that it is $\omega$-stable with prime model $\af_2$.
\end{abstract}

\maketitle
\section{Introduction}
The free factor complex $\af_n$ is a simplicial complex naturally associated to the free group $\F_n$ of rank $n$. Its vertices are the proper free factors of $\F_n$ and edges denote containment. The cliques in this graph form simplices, making this graph a simplicial complex of rank $n-1$ very similar in flavor to the spherical Tits building associated to a projective space. Free factor complexes were originally introduced by Hatcher and Vogtmann \cite{HV}, who showed that (as in the case of spherical buildings arising from projective spaces), the free factor complex of a free group of rank $n$ is homotopy equivalent to a wedge of spheres of
dimension $n-2$.
Recently, Bestvina and Bridson \cite{BB} extended this analogy to the fundamental theorem of projective geometry by showing that for $n\geq 3$ we have $\Aut(\af_n)\cong \Aut(\F_n)$. In the case of rank $2$, all proper free factors have rank $1$ and so there is no containment relation. In this case, we study the graph of the complement relation between the free factors, i.e. we put an edge between free factors if their free product is the free group. It was shown in \cite{BB} that this complement relation is definable in the free factor complex of higher rank. Thus, necessarily, in order to study the model theory of $\af_n$ for $n\geq 3$ we have to start with the model theory $\af_2$ with respect to the complement relation and this is what we do in this note.

\section{The Free Factor Complex of rank 2}\label{sec:background}
\begin{definition}
A subgroup \( H \leq \mathcal{F}_2 \) is called a \emph{free factor} of \( \mathcal{F}_2 \) if there exists a subgroup \( K \leq \mathcal{F}_2 \) such that
\[
\mathcal{F}_2 = H * K,
\]
where \( * \) denotes the free product of groups. In this case, \( \mathcal{F}_2 \) is the free product of \( H \) and \( K \), and both \( H \) and \( K \) are themselves free groups.

\end{definition}

\begin{remark}
An element $ g \in \F_2 $ is called \emph{primitive} if and only if there exists a basis $ \{x, y\} $ of $\F_2 $ such that $g = \langle x\rangle$ if and only if $\langle g \rangle $ is a nontrivial proper free factor of $\F_2$. Thus, the nontrivial proper free factors of $\F_2$ are precisely the infinite cyclic subgroups generated by primitive elements. 
\end{remark}

\begin{definition}
    The free factor complex $\af_2$ is the simplicial graph 
    whose vertices are the nontrivial proper free factors of $\mathcal{F}_2$ and two distinct 
    vertices $x, y\in\af_2$ are connected by an edge $E(x,y)$ if they generate $\mathcal{F}_2$ as a free product.
\end{definition}

\begin{notation} 
     To simplify notation we will mostly write $x\in\af_2$ instead of $\langle x \rangle \in \af_2 $ for a primitive element $x \in \F_2$.
\end{notation}

The following are some easy, but useful observations:
\begin{remark}
Edges are invariant under $\Aut(\F_2)$, $Aut(\F_2)$ acts transitively on the edges of this graph and the stabilizer of an edge  $(x,y)$ in $Aut(\F_2)$ is isomorphic to $\Z_2\times\Z_2$.
Note that if $E(x, y)$ holds, then $x$ and $y$ are not conjugate. 
\end{remark}

We have the following characterization of edges in $\af_2$:
\begin{remark}\label{rem:bases}% (see e.g. \cite{BB} Lemmas 2.1 and 2.2 ) 
 Let $\F_2=\langle a,b\rangle$. Then $E(\b,\x)$  holds for $\x\in\af_2$
if and only if $x=b^m a^\d b^k$ for some $m, k\in\Z$ and $\d\in\{\pm 1\}$.
\end{remark}

We will be using the quantifier-free diagram of specific subgraphs $\af_2$ in our axiomatization as the main building blocks, see Section~\ref{sec:appendix}. These subgraphs are inductively defined as follows:

\begin{definition}\label{rem:sticks}\label{def:extension}
 Let $e = \{ x, y \}$ be an edge in $\af_2$. 
   There exist exactly four vertices, called \emph{sticks} in \cite{BB}, which are adjacent to both of $x$ and $y$, namely the vertices generated by $ xy, xy\inv, x\inv y $ and $x\inv y\inv$, respectively.  We call the graph consisting of these six vertices the $1$-extension  $\Exp_1(e)$ of $e$ and put $\Exp_0(e)=e$.
  
Inductively, the $k$-\expa $\Exp_k(e)$ of $e$ is $\Exp_{k-1}(e)$
        together with the $1$-extension of every edge in the $k-1$-extension.        
   Then \[\af_2=\bigcup_{k\in\N} \Exp_k(e).\]   
    If $A=\Exp_k(e)$ for some edge $e$, we call $A$ a block of level $k$.      
\end{definition}

The Farey graph $\of_2$ is the quotient of $\af_2$ under the conjugation action of~$\F_2$.
We briefly recall its inductive definition:
\begin{definition}\label{def:Farey graph}(see e.g.\cite{MT}) The Farey graph $FG$ is the planar graph defined inductively as follows: let $FG_0$ be the graph consisting of two triangles sharing one edge.
If $FG_i$ has already been defined, $FG_{i+1}$ is the graph obtained by adjoining a triangle to every boundary edge of $FG_i$. Finally $FG=\bigcup_{i\in\N}FG_i$.
\end{definition}

Note that in contrast to the Farey graph, the extensions $\Exp_k(e)$ for $k\geq 2$ are not free: some of the new vertices added in the 1-\expa of $\Exp_k(e)$ have valency strictly larger than two, i.e. there are edges to vertices in $\Exp_{k-1}(e)$ and these additional edges are controlled by Remark~\ref{rem:bases}, see Section~\ref{sec:appendix}.

\begin{remark}\label{fct:edge_acl}
\begin{enumerate}
\item
Clearly, for any edge $e=E(x,y)$ we have  $\acl(\{x, y \}) = \af_2$  and in fact $\af_2\subset\dcl(Exp_1(e))$.

    \item    The initial edge of a block of each level is unique. This can be seen inductively by removing all edges that are not contained in exactly four triangles.
      \item 
\label{rem:blocks are isomorphic}
    Any two blocks of level $k$ are isomorphic under an automorphism of $\af_2$.
\end{enumerate}
\end{remark}

\section{An additional graph structure on $\af_2$}\label{sec:C1 graph}

Until further notice we will be working in the language of graphs $L_0=\{E\}$.

\begin{notation}

For $x, y\in\af_2$ we write $D_k(x,y)$ if $k$ is minimal such that there is a path of length $k$ from $x$ to $y$ and
\[D_k(x)=\{y\in \af_2\colon D_k(x,y)\}.\]
\end{notation}

We will be expanding the language by $L_0$-definable relations and start with conjugates.
We use Remark~\ref{rem:bases} to describe the set of conjugates of a vertex in $\af_2$. 
To this end we define an additional graph structure on $\af_2$:
\begin{definition} For $x,y\in \af_2$ we put an edge $C(x,y)$ and say that $x$ is a 1-conjugate of $y$ if there is a basis $\{x,z\}$ for $\F_2$ such that $y=x^z$.
\end{definition}

To show that this edge relation $C$ is definable in $L$ we first observe:

\begin{lemma}\label{lemma:conjugates intersection}
    If $x, y, z \in \af_2$ are distinct conjugates, then \[|D_1(x) \cap D_1(y) \cap D_1(z)| \leq 1.\]
\end{lemma}

\begin{proof}
Suppose towards a contradiction that we have distinct conjugates $x, y, z\in \af_2$ and 
    $t_1\neq t_2 \in D_1(x) \cap D_1(y) \cap D_1(z)$.
    By Remark~\ref{rem:bases} we may assume that $x = \a, t_1=\b$. Then by Remark~\ref{rem:bases} 
    we have  $y = \langle  a^{b^m} \rangle$, and
    $z= \langle a^{b^k}\rangle$ for some $m, k\in \Z$.   
   Again using Remark~\ref{rem:bases} we obtain from $E(\a, t_2)$ that there are $m_1, m_2 \in \mathbb{Z}$ such that $t_2 = \langle a^{m_1} b a^{m_2} \rangle$ and from $E(t_2, y)$ we see that $y= a^{t_2^l}=a^{b^m}$. This implies $t_2=a^{m_1}b$ and $m=l$.
    
    Applying the same argument to $z$ we now see that $z=a^{t_2^m}=a^{b^m}=y$, contradicting our assmption. 
\end{proof}

We can now show that the relation $C$ is definable.

\begin{proposition}\label{the_best_lemma}
    Let $x, y \in \af_2$. The following are equivalent:
    \begin{enumerate}
        \item $x$ is a 1-conjugate of $y$;
        \item $D_1(x) \cap D_1(y)$ is infinite;
        \item $|D_1(x) \cap D_1(y)| \geq 5$;
    \end{enumerate}

If any of the above conditions hold, we call the set
    \[    \ell{(x, y)} := D_1(x) \cap D_1(x')    \]
  the \emph{line} determined  by $x$ and $y$.

\end{proposition}

\begin{proof}

    $(1) \Rightarrow (2)$: Let $y$ be a 1-conjugate of $x$ with respect to $t$, so we can assume that $y =  t x t^{-1} $. Then $t \in D_1(\x) \cap D_1(\y)$ and for every $k \in \mathbb{Z}$, $\langle t x^k \rangle \in D_1( \langle x \rangle ) \cap D_1( t x t^{-1} )$, so this set is infinite.

    $(2) \Rightarrow (3)$: This is clear.

    $(3) \Rightarrow (1)$: Let $x, y\in\af_2$ be such that  $|D_1(x) \cap D_1(y)| \geq 5$.
    We first claim that $x$ and $y$ are conjugate.
    Suppose not and let $\pi$ be the quotient map from $\af_2$ to $\mathcal{OF}_2$ taking
    every vertex to its conjugacy class.
    If $x$ and $y$ are not conjugate, $\pi(x) \neq \pi(y)$.
    Since $\of_2$ is isomorphic to the Farey graph and two vertices 
    in the Farey graph have at most two neighbors in common, we have
     $|D_1(\pi(x)) \cap D_1(\pi(y))|\leq 2$.
     By the pigeonhole principle
    there exist $A \subset D_1(x) \cap D_1(y)$ such that $|A| \geq 3$ 
    and $|\pi(A)| = 1$. 
    Thus $A$ consists of conjugates having at least two common neighbours, 
    contradicting Lemma~ \ref{lemma:conjugates intersection}.
    Hence,  $x$ and $y$ are conjugate. 
    
    Now let $t\neq t' \in D_1(x) \cap D_1(y)$. 
    Without loss of generality we may assume that  $x =\b$ and $t =\a$,
    so by Remark \ref{rem:bases} we have $y = a^{-k} b a^k$ for $k \in \mathbb{Z}$.
    Since $t' \in D_1(\b)$, by Remark~\ref{rem:bases} we have
    $t' = \langle b^{k_1} a b^{k_2} \rangle$ for some  $k_1, k_2 \in\Z$, not both zero,
    and since $t'\in D_1(y)$ we also have 
    $t'=y^{m_1}ay^{m_2}=(a^{-k} b a^k)^{m_1} a (a^{-k} b a^k)^{m_1}$ for some $m_1,m_2\in\Z$.
    Since the reduced word for $t'$ begins or ends in the letter $b^{\pm 1}$, it follows that $k=0$
    and hence $y=b^a$ as claimed.

\end{proof}

We need the following lemma:

\begin{lemma}\label{lem:conjugates}
Let $\F_2=\langle a, b\rangle, g\in \F_2$. Then  $C(b^g,y)$ holds if and only if $y=b^{a^\d b^kg}$ for some $k\in\Z$ and $\d\in\{\pm 1\}$.
\end{lemma}
\begin{proof}
First consider the case $g=1$. If $C(b, y)$ holds, then by Remark~\ref{rem:bases} we know that $y=b^x$ where $x=b^m ab^k$ for some $m,k\in\Z$ and  hence $y=b^{ab^k}$. The converse is clear.

By conjugating the previous argument we now see that if $C(b^g,y)$ holds, then $y=(b^g)^x$ where $x$ is of the form $x=(b^g)^m a^g (b^g)^k$ for some $m,k\in\Z$. Hence $y=(b^g)^{a^g(b^g)^k}=b^{ab^kg}$ as claimed. Again the converse is clear.
\end{proof}

\begin{corollary}\label{cor:C1 trans}
\begin{enumerate}
\item $Aut(\af_2)$ acts transitively on $C$-edges.

\item $\Aut(\af_2)$ acts transitively on triples $(x,y,z)$ with $C(x,y)$ and $z\in\ell(x,y)$.
\end{enumerate}

\end{corollary}
\begin{proof}
For the first part consider bases $\{a,b\}$ and $\{c,d\}$ for $\F_2$ and $C$-edges $C(b,b^{a^{\d_1} b^m})$ and $C(d,d^{c^{\d_2} d^k}), k, m\in\Z, \d_1,\d_2\in\{\pm 1\}$. If $\d_1=\d_2$ let $\sigma\in\Aut(\F_2)$ with $\sigma(a,b)=(c,d)$. Otherwise let  $\sigma(a,b)=(c\inv,d)$. Let $\kappa$ be conjugation by $d^{k-m}$. Then $\sigma(b,b^{a^{\d_1} b^m})=(d,d^{c^{\d_2} d^m})$ and hence the composition of $\sigma$ with $\kappa$ takes $(b,b^{a^{\d_1} b^m})$ to $(d,d^{c^{\d_2} d^m})$.

 Part (2): By part (1) it is enough to show that the stabilizer of  $x=\b$ and $y= \langle aba^{-1}\rangle$ acts transitively on  $\ell(x,y)= \{ ab^{k} \,\mid \,  k \in \mathbb{Z} \}$. Let $z=ab^k\in\ell(x,y)$ and note that $\{b,ba\}$ and $\{b,b^ka\}$ are bases. Let $\sigma\in\Aut(\F_2)$ be such  that $\sigma(b,ba)=(b,b^ka)$. Then
 $\sigma(a)=b^{k-1}a$ and hence $\sigma$ fixes $b$ and $b^a$ and takes $ba$ to $b^ka$ as required.
\end{proof}

For future reference we also note the following:
\begin{remark}\label{rem: fix nbhs and act on lines}
 Let $\sigma\in\Aut(\F_2)$ be such  that $\sigma(a, b)=(b^{k-1}a,b)$ as above. Then
 $\sigma$ fixes $b, b^a$ and hence  $b^{ab^k}$ for any $k\in\Z$. 
\end{remark}

\begin{notation}
Given $g$ in reduced form 
\[g=a^{k_1} b^{m_1}\ldots a^{k_s} b^{m_s}, k_i, m_i\in\Z, 1\leq i\leq s \]
we call $k=\Sigma_{i=1}^s |k_i|$.
 the $a$-length of $g$.
 
We call each $b^k, k\in Z,$ between two occurrences of $a^{\pm 1}$  a $b$-tile (note that $b$-tiles can be empty). Thus, if $g$ has $a$-length $k$, then there are $k+1$-many $b$-tiles, including the ones before the first and after the last occurrence of  $a^{\pm 1}$.

For $x, y\in\af_2$ we say that $x$ is a $k$-conjugate of $y$ (and write  $C_k(x,y)$) if $k$ is minimal such that there is a reduced $C$-path of length $k$ from $x$ to $y$.
We write \[C_k(x)=\{y\in\af_2\colon C_k(x,y)\}\mbox{\ and \ } C(x)=\bigcup_{k\in\N} C_k(x).\]

\end{notation}

\begin{proposition}\label{prop:conjugacy tree}
Suppose $x\in\F_2=\langle a,b\rangle$ is conjugate to $b$, i.e. $x=b^g$ where  we write $g$ in reduced form as
\[g=a^{k_1} b^{m_1}\ldots a^{k_s} b^{m_s}, k_i, m_i\in\Z, 1\leq i\leq s \]

Then there is a unique $C$-path from  $\x$ to $\langle b\rangle$ of length $k=\Sigma_{i=1}^s |k_i|$.
Hence the conjugacy class of a vertex in $\af_2$ is a tree under the edge relation $C$.

\end{proposition}
\begin{proof}
The proof is by induction on $k\geq 1$.
First suppose $k=1$, so $g=a^\d b^m$ for some $m\in \Z$ and $\d\in\{{\pm 1}\}$.  Then $C(b, b^g)$ holds by Lemma~\ref{lem:conjugates}.

Now suppose we have proved the claim for $k\geq 1$ and suppose $g$ has $a$-length $k+1$. Write
$g=a^\d h$ where the $a$-length of $h$ is $k$  and $\d\in\{{\pm 1}\}$. By induction assumption there is a $C$-path of length $k$ from $b$ to $b^h$ and
 by Lemma~\ref{lem:conjugates} we see that $C(b^h,b^g)$ holds. Hence  we obtain a $C$-path of length $k+1$ from $b$ to $b^g$ and the uniqueness of this path follows from the fact that these expressions are unique.
\end{proof}

 \begin{remark}\label{rem:line_intersection}
 Note that by Proposition \ref{lemma:conjugates intersection}, the lines $\ell(x,y_1)$ and  $\ell(x,y_2)$  have at most one element in common since  $\ell(x,y_1) \cap \ell(x,y_2)= D_1(x)\cap D_1(y_1)\cap D_1(y_2)$ and $x, y_1, y_2$ are conjugate.   
 \end{remark}

\begin{definition}
We say that the lines $\ell(x,y_1), \ell(x,y_2)$        are  
        \emph{parallel} if they do not intersect and \emph{orthogonal} if they intersect in exactly one vertex. We call this vertex the \emph{witness} to orthogonality.
        
Similarly, if $y_1\in C_2(y_1)$, 
       we say that \( y_1 \) and \( y_2 \) are parallel (orthogonal, respectively) if \( l{(x, x_1)} \) and \( l{(x, x_2)} \) are parallel (orthogonal, respectively) where $x$ is the unique vertex in $C_1(y_1)\cap C_1(y_2)$. 
       
\end{definition}

\begin{remark}\label{rem:ort transitive} Note that  by Corollary~\ref{cor:C1 trans} (2) there is a one-to-one correspondence between $z\in\ell(x,y)$ and the unique line $\ell(y,w)$ containing $z$. It thus follows that  $\Aut(\af_2)$ acts transitively on triples $(x,y,w)$ with $x, w\in C_1(y)$ where $x$ and $y$ are orthogonal.
\end{remark}

Being parallel is an equivalence relation on $C_1(x)$ for $x\in\af_2$.:

\begin{lemma}\label{lem:parallel equiv}
    Write $C_1(\b)= B_1\dot\cup B_2$ where
     \[
    B_1 = \left\{ \langle b^{a b^m}\rangle \mid m \in \mathbb{Z} \right\}, \quad
    B_2 = \left\{ \langle b^{a\inv b^m}\rangle \mid m \in \mathbb{Z} \right\}.
    \]
Then $x,y\in C_1(b)$ are parallel if and only $x,y$ belong to the same $B_i, i=1,2$. 
\end{lemma}

\begin{proof}
Let $x\in B_1$, so $x=\langle  b^{a b^m}\rangle$ for some $m\in\Z$ by Lemma~\ref{lem:conjugates}. It suffices to show that if $y\in\ C_1(b)$ is orthogonal to $x$, then $x\in B_2$.
So let $t\in \ell(\b, x)$. Then $t =\langle b^s a\inv b^m\rangle=\langle b^{-m} a b^{-s}\rangle$ for some $s \in \mathbb{Z}$ by Remark~\ref{rem:bases}.
  By the same remark, if $y\in C_1(\b)$ is such that $t\in\ell(\b,y)$, then we have $E(t,y)$ and hence $y =\langle  b^{a b^{-s}} \rangle\in B_2$, proving the claim.
\end{proof}

\begin{definition}
    We call parallel lines $\ell{(x, x_1)}$ and $\ell{(x, x_2)}$ neighbors if 
    for every vertex $ y_1 \in\ell{(x, x_1)} $,
    there is some $y_2\in\ell{(x, x_2)} $
    with $E(y_1,y_2)$.
\end{definition}

\begin{lemma}\label{lmm:finite-neighbor}
Every line $\ell$ has exactly four neighbours. More precisely, $\ell(\b, b^{ab^k})$ is a neighbour of $\ell=\ell(\b, b^a)$ 
if and only if $k\in \{1,-1\}$.
\end{lemma}
\begin{proof}
 By Corollary~\ref{cor:C1 trans}  and Lemma~\ref{lem:parallel equiv}  it suffices to prove the second statement.
 
By Corollary~\ref{cor:C1 trans} and Remark~\ref{rem: fix nbhs and act on lines} the stabilizer of $b^{ab^k}, k\in\Z$, acts transitively (and hence regularly) on the points of each line $\ell(b,b^{ab^k}), k\in \Z$. Therefore it suffices to consider $\a\in\ell(b,b^a)$ and show that there are exactly two vertices in $z\in D_1(\a)\cap \ell(b,b^{ab^k})$ for $k\in\Z$ and in this case $k\in\{1,-1\}$.

Note that if $z\in D_1(a)\cap D_1(b)$, then by Definition~\ref{rem:sticks} we have $z=a^\delta b^\epsilon$ where $\epsilon,\delta\in\{-1,1\}$.
If furthermore $z\in D_1(b^{ab^k})$, then 
\[z^{b^{-k}a\inv}=(a^\delta b^\epsilon)^{b^{-k}a\inv } =ab^k a^\delta b^\epsilon b^{-k}a\inv  \in D_1(b)=\{b^m ab^n\colon m, n\in\Z\}\] by Remark~\ref{rem:bases}.
Thus we must have $\delta=1$ and $k=\epsilon$. This proves the claim.   
\end{proof}

\begin{definition}\label{def:nbh cl}

We call a subset $X$ of $\af_2$ neighbour closed if for any line $\ell(x,y)$ with $x,y\in X$ and neighbours $\ell(x,y_1), \ell(x,y_2)$
    we have $y_1,y_2\in X$.
    
We write $\cl(x,y)$ (or $\cl(\ell)$ if $\ell=\ell(x,y)$) for the smallest neighbour closed set containing $x,y$ with $C(x,y)$. 

We say that $\cl(\ell)$ and $\cl(\ell')$ are orthogonal families if $|\cl(\ell)\cap \cl(\ell')|= 1$  and there exist $x_1 \in \cl(\ell)$ and $x_2\in\cl(\ell')$ such that $x_1$ and $x_2$ are orthogonal.

We call a $C$-path \emph{straight} if it is contained in $\cl(x,y)$ for some $x,y$ with $C(x,y)$.
\end{definition}

The following observations will be useful
\begin{remark}\label{rem:useful things about nbh closed}

\begin{enumerate}
    \item By symmetry all four neighbours of $l(x,y)$ belong to $X$ and with the notation of Lemma~\ref{lem:parallel equiv} we have $B_i\subseteq \acl(\b,x)$ for any $x\in B_i, i=1, 2$.

    \item If $y_1,y_2\in \cl(x,y)$ with $C(y_1,y_2)$, then it follows from the definition that $\cl(x,y)=\cl(y_1,y_2)$.
    
     \item By definition, $\cl(x,y)$ is $C$-connected and $\cl(x,y)\subseteq\acl(x,y)$. Hence $\cl(x,y)\subseteq\acl(x',y')$ for any $x',y'\in \cl(x,y)$.
 
 \item \label{cl description} Let $\ell=\ell(x,y)$, let $z\in C(x)=C(y)$ and let $p=(z=x_0,\ldots, x_k=x, x_{k+1}=y)$ be a $C$-path from $z$ to $y$. Then $p$ is straight if and only if $x_i$ and $x_{i+2}$ are parallel for all $i=0,\ldots k-1$. 
 
 \item  By the previous part of the remark for lines $\ell, \ell'$ with $\cl(\ell)\neq\cl(\ell')$ we have $|\cl(\ell)\cap \cl(\ell')|\leq 1$.

\end{enumerate}
\end{remark}

\section{Paths avoiding conjugacy classes}\label{sec:avoiding paths}

We say that a path $(x_0,\ldots, x_k)$ avoids a set $X$ if $x_i\notin X$ for all $i=0,\ldots, k$.
We first observe the following:

\begin{lemma} \label{lmm:block-minues-conjugates}
    Let $A=\Exp_k(e)$ for some edge $e\in\af_2$ and $x \in A$. Then any $v_1,v_2\in A\setminus C(x)$ are connected by an $E$-path avoiding $C(x)$ of length at most $2k$.
\end{lemma}

\begin{proof}
    We prove this by induction on $k$. Since $\Exp_0(e)=\{e\}$ the claim clearly holds for $k=0$. Now suppose the claim holds for $k-1$ and let  $A'=\Exp_{k-1}(e) \subset A$. By construction (see Section~\ref{sec:appendix}), any vertex in $ A \setminus A'$ is a stick of some edge in $A'$, i.e. it has  edges to at least two vertices in $A'$ forming a triangle. Clearly, not both of these vertices can belong to $B$. Hence for any $v_1,v_2\in A\setminus B$ there are $E$-neighbours
   $u_1, u_2\in A'\setminus B$ and 
    by the induction hypothesis, there exists a path of length at most $2(k-1)$ between $u_1$ and $u_2$ in $A' \setminus (B \cap A')$, thus yielding a path of length at most $2k$ from $v_1$ to $v_2$.
\end{proof}

The next lemma is a technical tool for the proof of Lemma~\ref{lem:crucial} and Corollary~\ref{main-prop}. We refer to  the definition of the Farey graph in \ref{def:Farey graph} above:

\begin{lemma}
Let $P_{ab}$ be the set of cyclically reduced primitive words in $\F_2=\langle a, b\rangle$ starting with $a$.
There exists a map \[f:\of_2\setminus\{[b]\}\longrightarrow P_{ab}\] labeling the vertices of $\of_2\setminus\{[b]\}$  by elements in $P_{ab}$ such that the following holds:

\begin{itemize}
\item $f([x])$ is  a generator for a representative for $[x]\in\of_2$, i.e. $[x]=[\langle f([x])\rangle ]$;

\item $f([a])=a$ and $f([x])$ starts with $a$, does not contain any $a\inv$, and ends in $b$ or $b\inv$;

\item for $x,y\in\af_2$ with $E(x,y)$ there exists $k\in\Z$ such that 
\[\mbox{if\ }x=\langle f([x])\rangle \mbox{ \ then \ } y = \langle f([y])^ {x^k} \rangle.\]
\end{itemize}
Finally, we put $f(\b)=b$.
\end{lemma}
\begin{proof}
We inductively define a labeling on the vertices of the Farey graph different from $[b]$ and an orientation on the edges. We write the orientation of an edge as the ordered pair of labels of the initial and the terminal vertex.  Edges appearing in subsequent levels are oriented in the unique way so that composing the two new edges yields a path from the inital vertex of the edge to its end vertex.

More precisely, we start the inductive definition on the graph of level 1 and label the inital vertices of the directed Farey graph as $a$ and $b$  and the two vertices of level 1 by $ab$ and $ab\inv$, respectively. We orient the corresponding edges as $(a,b), (a,ab), (ab,b), (a,ab\inv)$ and $(ab\inv, b\inv)$.\footnote{Note that in this way the label on $[b]$ is not uniquely defined.}

We extend this labelling inductively as follows: if $f$ has been defined on $A_k$ and $z\in A_{k+1}$ arises from an oriented edge $(x,y)$, we put $f(z)=f(x)f(y)\in\F_2$, thus defining it on all of $\of_2$.

It follows inductively that property (2) holds and that the image of $f$ is contained in $P_{ab}$.

    To see that (3) holds, we note the following immediate consequence of the construction:  if the vertex labeled $z=uv$
    arises 
    from the oriented edge $(u,v)$, then the edges starting in $z$ are exactly those 
    whose terminal vertices are of the form $z^kv, k\geq 0$ and the edges ending 
    in $z$ are exactly those whose initial vertices are of the form $uz^k, k\geq 0$.
   
   Now consider $x,y\in\af_2$ with $E(x,y)$. By conjugating we may assume that $f([x])=x$ and that $x$ arises from a directed edge $(u,v)$. Then by the previous consideration $f([y])$ is  either $x^kv$ or $ux^k$ for some $k\geq 0$ and hence $y$ is conjugate to the corresponding element.
   By symmetry assume $y=(ux^k)^h$ for some $h\in\F_2$. By Remark~\ref{rem:bases} applied to $\{u,x\}$ and $\{y,x\}$, we see that $y=x^m u x^l=(ux^k)^h$ for some $m, e\in\Z$ and hence $y=f([y])^{f([x])^{-m}}$ as required.
\end{proof}

An obvious induction now shows the following result:

\begin{corollary}\label{lmm:conjugates-over-paths}
Let $ x, y \in \af_2$  
and \[ p = (x = t_0, \dots, t_k = y) \] be a path in $\af_2$ from $x$ to $y$. If $x = \langle  f(x) \rangle$, then there exist integers $ s_1, \dots, s_k \in \Z$ such that $y=\langle f(y)^\lambda\rangle$ where
\[ \lambda= f(t_k)^{s_k} \ldots f(t_0)^{s_0} .
\]

\end{corollary}

We invoke the following fact (which can be read off from the definition of $f$):

\begin{fact}\cite{Cohen1981}\label{fct:making a basis}
    Let $x\neq a^{\pm 1}, b^{\pm 1}$ be a cyclically reduced primitive element in \(\mathcal{F}_2 = \langle a, b \rangle \). Then 
    there exists $k > 0$ and $\epsilon \in  \{\pm 1 \}$ such that 
    one of $x$ or $x^{-1}$ has a cyclic permutation of the form
    \[a^{m_1}b^{n_1} \dots a^{m_s}b^{n_s}\] where either
 \[
 \begin{gathered}
 m_1 = m_2 = \dots = m_s = \epsilon, \\
 \{ n_1, \dots, n_s \} = \{ k, k+1 \}
 \end{gathered}
 \]
 or, with the roles of \( a \) and \( b \) interchanged,
 \[
 \begin{gathered}
 n_1 = n_2 = \dots = n_s = \epsilon, \\
 \{ m_1, \dots, m_s \} = \{ k, k+1 \}.
 \end{gathered}
 \]
\noindent 
We call $k$ the tile length of $x$ and write $\tile_b(x)=k$  if $m_i\in\{k, k+1\}, i=1,\ldots, q$. 
\end{fact}

Using the previous fact and the construction of the map $f$ we conclude:

\begin{corollary}\label{cor:avoiding path length}

If $[x],[y]\neq [b]$ in $\of_2$ form an edge, then $|\tile_b(f[x])-\tile_b(f([y])|~\leq~1$.

 If for $x,y\in\af_2$ there is a path of length $k$ from $x$ to $y$ avoiding $C(b)$, then $|\tile_b (x) - \tile_b (y)|\leq k$. 
\end{corollary}
\begin{proof}
The first statement follows from the construction of $f$ since by Fact~\ref{fct:making a basis} the lengths of the $b$-tiles in any element of $P_{ab}$ can only differ $1$.
The second statement follows from the first by induction.
\end{proof}

Note that $[b]$ has edges to $[ab^k]$ in $\of_2$ for any $k\in \Z$, so the assumptions in the previous corollary are necessary.

 \begin{remark}\label{changing_tile}
Recall that elements in $P_{ab}$ start with $a$, do not contain any $a\inv$ and end with $b$ or $b\inv$. The following observations follow directly from these properties.  Let $w\in\F_2$.
\begin{enumerate}
\item \label{rmk2} Suppose $w$ has a reduced expression $w= y a^{\pm 1} b^k$ for some  $y\in\F_2, k\in\Z,$ and $t\in P_{ab}\cup P_{ab}\inv$ is such that $|\tile_b(t)-k| \geq 2$.  Then in $wt$ all $b$-tiles of $y$ remain unchanged.

\item  \label{rmk1} Suppose $w$ has a reduced expression $w=y_0a^{\pm 1}b^{\pm k}a b^{k+m}, m\geq 2$ and $t\in P_{ab}\inv$. Then $wt$ has a reduced expression $wt=y_0 y_1$ for some word $y_1$. Furthermore, if the first $b$-tile after $y_0$ is changed, then $\tile_b(t)\in \{k+m, k+m-1\}$ and hence the new $b$-tile after $y_0$ has length at most $2k+m+1$.

\end{enumerate}

\end{remark}

\begin{corollary}\label{cor:preparation}
Let $p=(t_0=\a, t_1,\ldots, t_k)$ be a path avoiding $C(b)$. Then for any $s_0,\ldots, s_k\in \Z$ the lengths of the $b$-tiles in $\lambda=t_0^{s_0}\ldots t_k^{s_k}$ is bounded by $\sum_{i=1}^{k+1} i\leq k^2$.
\end{corollary}
\begin{proof}
This follows at once from  Corollary~\ref{cor:avoiding path length} and Remark~\ref{changing_tile} (\ref{rmk2}).
\end{proof}

\begin{definition}
    For $y \in \af_2$, we say that the path $p = (x_0, \dots, x_k)$ avoids $C_m(y)$ if $x_i \notin \bigcup_{j \leq m} C_j(y)$ for $ 0 < i < k$. 
\end{definition}

\begin{lemma}\label{lem:crucial}
Let $p = (a=t_0, \dots, t_k=z)$ be a path of length $k$ avoiding $C_m(b)$ with $m \geq k$.
Let $s_0, \dots, s_k \in \mathbb{Z}$ such that for $
\lambda = f(t_0)^{s_0} \cdots f(t_k)^{s_k}$
we have
 $z=\langle f(z)^{\lambda\inv}\rangle$.
Then the lengths of the first $(m + 2 - k)$-many $b$-tiles in $\lambda$ are bounded by
\[
m_k=\sum_{i=0}^k 2^i.
\]

\end{lemma}

\begin{proof}
 
We prove the lemma by induction on $k$. Write \[
\lambda_i= f(t_0)^{s_0} \cdots f(t_i)^{s_i}.
\]

For $k=0$ the statement is obviously true. So suppose $0<k\leq m$ and assume the claim holds for all $k'<k$.  Then by the induction hypothesis for $k-1$ the first $(m + 3 - k)$-many $b$-tiles in $t_{k-1}=\langle f(t_{k-1}^{\lambda_{k-1}})\rangle$ are bounded by $m_{i-1}$.

If $t_k\in C(b)$, then $f(t_i)=b$, and since $p$ avoids $C_m(b)$ we see from Lemma~\ref{lem:conjugates} that $\lambda_{k-1}$ has $a$-length at least $m+1$, so contains at least $m+1$ many $b$-tiles.
By Remark~\ref{changing_tile} only the last $b$-tile of $\lambda_{k-1}$ can change in $\lambda_k$ and hence the first  $(m + 2 - k)$-many $b$-tiles in $\lambda$ are still bounded by $m_{k-1}<m_k$ as claimed.

Now suppose $t_k\notin C(b)$. If $\lambda_{k-1}$ has $a$-length at least $m+1-k$, then by Remark~\ref{changing_tile} (\ref{rmk1}) and the induction hypothesis the first $(m+2-k)$-many $b$-tiles in $\lambda_k$ are bounded by $m_{k-1}+m_{k-1}+2=2m_{k-1}+2=m_k$.
However, if $\lambda_{k-1}$ has $a$-length less than $m+1-k$, then by the previous paragraph we see that $t_{k-1}\notin C(b)$. If $p$ avoids $C(b)$, then  by Corollary~\ref{cor:preparation} all $b$-tiles in $\lambda$ are bounded by $2\cdot\sum_{i=1}^k i$.
Otherwise let $s<k$ be maximal with $t_s\in C(b)$. Then $\lambda_{s-1}$ has a-length at least $m+1$. If $\lambda_{k-1}$ has $a$-length less than  $(m + i - k)$,  by Remark~\ref{changing_tile} (\ref{rmk1}) there is $j < k$ such that $\tile_b(t_j) \leq m_{j-1}$. By Corollary~\ref{cor:avoiding path length} we have $\tile(f(t_k))\leq m_{k-2}+1$ and hence in this case all $b$-tiles in $\lambda$ are bounded by $m_{k-2}+1\leq m_k$.
This finishes the proof.
\end{proof}

The previous preparation implies the following:

\begin{corollary}
 \label{main-prop}

If $x, y \in D_1(\langle b \rangle )$ and there exists a path from $x$ to $y$ of length
    at most $k$ avoiding $C_k(b)$, then $y\in\Exp_{k'}(x, \langle b\rangle)$ for some $k'$ depending only on $k$.
\end{corollary}

\begin{proof} 
For (1) we may assume that $x =\a$. Then $y =  \langle b^m a b^{m'} \rangle$ for $m, m' \in \mathbb{Z}$ and it suffices to show that $m, m'$ are bounded by some number $m_0$ depending only on~$k$.
Suppose that $p = (\a=t_0, \dots, t_l =y)$ is a path from $\a$ to $y$
    for some $l \leq k$.
   By Corollary~\ref{lmm:conjugates-over-paths} there are $s_0,\ldots, s_l\in\Z$ such 
   that    $y=f(y)^{\lambda\inv}$ where $\lambda= f(t_0)^{s_0} \dots f(t_l)^{s_l}$.
     Since $f(\langle b^m a b^{m'} \rangle ) = a b^{m' + m}$,  we have 
    \[\lambda = b^{m'} (a b^{m' + m})^n \mbox{\quad for some\quad }n \in \Z.\]
   
    Therefore, it is enough to show that if $p$ avoids $C_k(b)$, there is a bound 
    on the lengths 
    of the first two $b$-tiles in $\lambda$ and this follows from Lemma~\ref{lem:crucial}.
\end{proof}

\section{Some bounds on blocks}\label{sec:block bounds}

Recall from Remark~\ref{fct:edge_acl}.\ref{rem:blocks are isomorphic} the following easy facts:

\begin{enumerate}
\item any two blocks of level $k\geq 0$ are isomorphic and there is an automorphism of $\af_2$ taking one to the other;
\item for any edge $e\in\af_2$ we have $\af_2\subseteq \dcl(\Exp_1(e))$.

\end{enumerate}
From these facts and the description in Section~\ref{sec:appendix}  we obtain the following immediate consequences:

\begin{proposition}\label{prop:easy bounds on blocks}
\begin{enumerate}
\item \label{cor:extension-change} for every $k\in\N$ there is some $k'=g_{\ref{cor:extension-change}}(k)$ such that for any $e'\in\Exp_k(e)$ we have $e\in\Exp_{k'}(e')$.
    
\item \label{lmm:common-edge}
     Suppose $B_1, B_2\subset \af_2$  are blocks of level $k_1$ and $k_2$, respectively,
      such that $B_1\cap B_2$ contains an edge $e$. 
      Then $B_1\cup B_2\subseteq \Exp_m(e)$
      where $m=g_{\ref{cor:extension-change}}(\max\{k_1,k_2\})$.

   \item \label{orth level}  any $x,y\in\af_2$ with $x$ and $y$ orthogonal are contained
    in a block of level $2$.
  
\end{enumerate}

\end{proposition}
\begin{proof}
Property (\ref{cor:extension-change}) follows from the previous facts, (2) is a consequence of (1) and  Property~(\ref{orth level}) uses Remark~\ref{rem:ort transitive}.
\end{proof}

Using the results from the previous section we also obtain the following consequences:

\begin{proposition}\label{prop: less trivial bounds on blocks} 
 \begin{enumerate}
 
   \item \label{lmm:not-conj}
   
     Suppose $B_1, B_2\subset \af_2$  are blocks of level $k_1$ and $k_2$, respectively, 
     with $x, y\in B_1\cap B_2$ and $x\not\in C(y)$. 
     There exists $k'$ depending only on $k_1, k_2$ such that for
     any $y_2\in D_1(x)\cap B_2$ the extension
     $\Exp_{k'}(x,y_2)$ contains $B_1$ and $B_2$.

   \item \label{lmm:not-neghbor-closed}
    Suppose $B_1, B_2\subset \af_2$  are blocks of level $k_1$ and $k_2$, respectively, 
    with $x, y\in B_1\cap B_2$. If $C(x)=C(y)$ and the $C$-path from $x$ to $y$ is not 
    straight, then there exists a block $B_3$ 
    of level $k'=g_{\ref{cor:extension-change}(\max\{k_1,k_2\})}$ such that $B_3$ contains $B_1, B_2$.
 
\end{enumerate}

\end{proposition}

\begin{proof}

(\ref{lmm:not-conj}):
    Let $x, y\in B_1 \cap B_2$ with $x\notin C(y)$. If $B_1\cap B_2$ contains an edge, 
     the claim follows from Lemma~\ref{lmm:common-edge}.
    Otherwise, pick $y_1\in D_1(x)\cap B_1$ and $y_2\in D_1(x)\cap B_2$.
    By Lemma~\ref{lmm:block-minues-conjugates}, for $i=1, 2$ there exist paths $p_i$ of length at most $2k_i$ from $y$ to $y_i$ contained in $B_i$ and avoiding $C(x)$. The composition of $p_1\inv$ and $p_2$ is a path from $y_1$ to $y_2$ of length at most $2(k_1+k_2)$ avoiding $C(x)$. 
    By Corollary~\ref{main-prop}, we have $y_2\in B'=\Exp_m(x,y_1)$ where  $m$ depends only on $2(k_1+k_2)$.
    Since $B'\cap B_1$ contains the edge $(x,y_1)$, by Proposition~\ref{prop:easy bounds on blocks}~(\ref{lmm:common-edge}), we find a block $B_1'$
    of level $m'= g_{\ref{cor:extension-change}}(m)$
   containing $B_1$ and $B'$. Applying the same argument to the edge $(x,y_2)$  in $ B'_1 \cap B_2$, we see that for $k'=g_{\ref{cor:extension-change}}(m')$ the block $\Exp_{k'}(x,y_2)$ contains $B'_1$ and $B_2$.

(\ref{lmm:not-neghbor-closed}):
   If $x,y\in B_1\cap B_2$ with $x\in C(y)$, the $C$-path $p = (x = t_0, \dots, t_m = z)$ from $x$ to $y$ is contained in $B_1\cap B_2$ by  Corollary~\ref{cor:blocks are convex}. If $p$ is not straight, let $ 0 \leq i < m - 1$ be
    such that $x_i$ and $x_{i+2}$ are orthogonal. By Corollary~\ref{cor:blocks are convex} the witness $z$ for the orthogonality also belong to both $B_1$ and $B_2$. Thus, $B_1\cap B_2$ contains an edge $e$ and hence $B_1, B_2$ are contained $\Exp_{k'}(e)$ where $k'=g_1(\max\{k_1,k_2\})$.
 
\end{proof}

\color{black}

The previous two propositions can be summarized as follows:
\begin{corollary}\label{cor:possible intersections}
Suppose $B_1, B_2\subseteq\af_2$ are blocks of level $k_1$ and $k_2$, respectively such that $X=B_1\cap B_2$ is not contained in a set of the form $\cl(\ell)$ for some line $\ell$. Then there exist $k'$ depending only on $k_1, k_2$ and a block $B_3$ of level $k'$ containing $B_1, B_2$.
\end{corollary}

\section{The Theory of the Free Factor Complex of rank 2}\label{sec:theory}

From now on we will be working in an infinite language $L$ containing the following binary predicate symbols:
\begin{itemize}
    \item  the relations $E$ and $C$,
    \item  relations $D_n, n\in\N$,
    \item  relations $C_n, n\in\N$,
    \item  relations $\para$ and $\ort$,
    \item  relations $N_k, k\in N$,
    \item  relations $B_k, k\in \N$.
\end{itemize}

\begin{remark}\label{rem:intended meaning}
The intended meaning of these predicates is as follows:
The relations $D_n(x,y), C_n(x,y)$ hold if and only if $x\in D_n(y)$ and $x\in C_n(y)$, respectively defined  as above. 
The relations $\para, \ort$ and $N_k$  denote the following
\begin{enumerate}
\item $\para(x_1,x_2)$ holds if there is some $x$ such that $\ell(x_1,x)$ and $\ell(x_2,x)$ are parallel;
\item $\ort(x_1,x_2)$ holds if  there is some $x$ such that $\ell(x_1,x)$ and $\ell(x_2,x)$ are orthogonal, i.e. $x_1, x_2\in C_1(x)$ and there is some $z\in D_1(x_1)\cap D_1(x)\cap D_1(x_2)$;
\item  $N_1(x,y)$ holds if $C_2(x,y)$ holds and for $z\in C_1(x)\cap C_1(y)$ the lines $\ell(z,x)$ and $\ell(z,y)$ are neighbours. Then $N_k(x,y)$ holds if there are lines $\ell_0=\ell(x,z),\ldots,  \ell_k(y,z')$ such that $\ell_{i-1}$ and $\ell_i$ are neighbours, $i=1,\ldots, k$.
\item  $B_k(x,y)$ holds if there is an edge $e$ such that $x,y\in\Exp_k(e)$.
\end{enumerate}

In this way we can consider $\af_2$ as an $L$-structure. Note that all predicates in $L$ are definable (with quantifiers!) in $L_0$.
\end{remark}

Let $T_L$ be the $L$-theory defining the relations of $L$ as given in the previous sections.

\begin{lemma}\label{lem:L-diagram determined by L0 diagram}

Let $e\in\af_2$ be an edge and $k\in\N$. Then there is some $k'$ depending only on $k$ such that the (quantifier-free) $L$-diagram of  $\Exp_k(e)$ is determined by the quantifier-free $L_0$-diagram of $\Exp_{k'}(e)$.
\end{lemma}
\begin{proof}
This is an immediate consequence of Proposition~\ref{prop:easy bounds on blocks}.
\end{proof}

\begin{definition}\label{def: theory of Af2}
Let $T$ be the $L$-theory extending $T_L$ axiomatized by the following axiom scheme:

\begin{enumerate}

    \item\label{no-isolated-vertex} every vertex is contained in an $E$-edge;
  
    \item\label{unique_extension} for every edge $e=E(x,y)$ and $k\in N$ there exists a unique $k$-extension $\Exp_k(e)$;
  
    \item \label{c_1-forest} there are no $C$-cycles of length $k, k\in \N$;
   
   \item \label{ort-or-par} if $C_2(x_1,x_2)$ holds, then exactly one of  $\ort(x_1,x_2)$ or $\para(x_1,x_2)$ holds.
  
   \item\label{para implies orth exists} if $\para(x_1,x_2)$ holds, then there exists $z$ with 
    $\ort(x_1,z)$ and $\ort(z,x_2)$;
    \item \label{parallel is equiv rel} if $x_1,x_2,y\in C_1(z)$ with $\para(x_1,x_2)$  and $\ort(y,x_1)$, then $\ort(y,x_2)$;

   \item\label{path avoiding k} if $x, y\in D_1(z)$ and there is an $E$-path of length at most $k$ from $x$ to $y$ avoiding $C_k(z)$, then $y\in \Exp_{k'}(x,y)$ where $k'$ depends only on $k$.

    \item\label{two-neighbors} For every $x, y$ with $C(x, y)$, there are exactly two vertices $z$ with $N_1(x,z)$ and  $ C(z, y)$;
    \item\label{block_expander} Suppose $B_1, B_2$ are blocks of level $k_1, k_2$, respectively, and there are $x,y\in B_1\cap B_2$ 
    which do not satisfy $N_m(x,y)$ for any $m\leq k$ for some $k$ depending only on $k_1$ and $k_2$, then there is a block $B_3$ of level $g(k_1,k_2)$ containing
    $B_1\cup B_2$. 
    \item\label{relations hold only in blocks} if $\ort(x,y), C(x,y)$ or $N_m(x,y), m\geq 1,$ holds, then $x,y$ are contained in a block of level $k$ where $k$ only depends on the relation $\ort, C_1, N_m$, respectively.
\end{enumerate}
\end{definition}

\begin{remark}\label{rem:T contained}

   Note that by the results in Sections~\ref{sec:background} -- \ref{sec:block bounds}, we have $T \subseteq Th(\af_2)$.
\end{remark}

\begin{definition}\label{def:block cover}
Let $M$ be a model of $T$ and let $p$ be a path in $M$. By Axiom~\ref{unique_extension} every edge of $p$ is contained in a unique copy of $\af_2$  and thus there are finitely many distinct copies $B_0,\ldots, B_k$ of $\af_2$ in $M$ such that $p$ passes through these blocks in the given order. We call $(B_0,\ldots, B_k)$ the block cover of $p$.

\label{def:non-returning}

We call $p$ non-returning if for all $i\leq k$ the intersection of $p$ with $B_i$ is connected.

\end{definition}
 
 By Axiom~\ref{unique_extension} the block cover of a path $p$ is uniquely determined.

Note that if two straight paths in $\af_2$ meet in a vertex, then they belong to orthogonal families and if a $C$-path is contained in two distinct copies of $\af_2$, then this path is straight by Axiom (\ref{block_expander}).

\begin{proposition}\label{prop:cycles}
Let $M$ be a model of $T$, let  $p=(x_0,\ldots, x_n=x_0)$ be a simple cycle in $M$ with block cover $(B_0,\ldots, B_k)$ and assume that $p$ is non-returning. Let $a_i\in p\ \cap B_{i-1}\cap B_i, i=0,\ldots, k$ with indices taken modulo $k$. Then  $C(a_i)=C(a_j)$ and $B_i\cap B_j\subset C(a_i)$ for all $1\leq i\neq j\leq k$ and for some $i\in\{1,\ldots, k\}$ there are orthogonal families $\cl(x,x_1),\cl(x,x_2)\in B_i\cap C(a_i)$ such that
$B_i\cap B_j\subseteq\cl(x,x_1)\cup\cl(x,x_2)$ for all $j\neq i$.

\end{proposition}
\begin{proof}
The proof is by induction on $k$, the statement being empty for $k=1$. Now consider $k>1$ and let $i\in\{1,\ldots, k\}$.
Let  $u_i, v_i\in D_1(a_i)$ with $u_i\in p\cap B_{i-1}, v_i\in p\cap B_i$.
Then a cyclic shift $p_i$ of $p$ is a path from $v_i$ to $u_i$ and since $u_i, v_i$ are not contained in a common block, it follows from  Axiom~\ref{path avoiding k} that $p_i$ passes through some $w\in p\cap C(a_i)$. Let $w\in p$ be closest to $a_i$ with this property. It follows from Lemma~\ref{lmm:block-minues-conjugates} that we may assume that $w\in p\cap B_{j-1}\cap B_j$ for some $j$. So we may assume $w=a_j\in C(a_i)$. 
Since the path from $a_i$ to $w$ avoids $C(a_i)$ we have $w\in\Exp_m(v_i,a_i)\subseteq B_i$ for some $m\in\N$ by Axiom~\ref{path avoiding k} and since $p$ is non-returning, we must have  $w=a_{i+1}$. Thus, for all $i\leq k$ we have $C(a_i)=C(a_{i+1})$, proving the first claim.

Now assume that the second claim does not hold. By changing the $a_i$ if necessary we may assume that the unique shortest $C$-path $q_i\subset B_i$ from $a_i$ to $a_{i+1}$ has $s_i\geq 3$ straight segments. By reordering if necessary, we may also assume that $s_0$ is minimal among the $s_i$.
Since $C(a_i)$ is a tree under the edge relation $C$, the path $q_0\subset B_0$ is equal to the composition of the $q_i, i=1,\ldots k$. Note that $q_1\circ q_2$ has at least $(s_0+2)$-many straight segments because by Axiom~\ref{block_expander}  at most the last segment of $q_1$ can cancel with the first segment of $q_2$.  Inductively we see that the composition of the $q_i, i=1,\ldots k$ has at least $s_0+(k-1)$ many segments. Since this composition is equal to $q_0$ it follows that $k=1$, a contradiction.
\end{proof}

\begin{definition}\label{def:admissible set}
We call an $L$-structure $A$ admissible with components $A_i, i\in J$ for some set $J$, if the following  conditions hold:
\begin{enumerate}
\item  $A=\bigcup_{i\in J} A_i$ where each $A_i$ is isomorphic as an $L$-structure to $\af_2$ or a block of level $k_i$ for some $k_i\geq 1$;
\item if $C(x,y), E(x,y), B_k(x,y)$, or $ N_k(x,y)$ holds for some $k\in\N$, then $x,y\in A_i$ for some $i\leq m$.

\item $C_k(x,y), k\geq 2,$ holds if and only if there are $x_0=x,\ldots, x_m=y$ such that for  $i=0,\ldots m-1$ we have $x_i,x_{i+1}$ in the same component and $C_{k_i}(x_i, x_{i+1})$ holds for some $k_i$ with $k=\sum_{i=0}^{m-1}k_i$;

\item if $C_2(x_1,x_2)$ holds for $x_1,x_2\in C_1(z)$ not in the same component and there is some $y\in A$ with $\ort(x_1,y)$ and $\ort(x_2,y)$, then $\para(x_1,x_2)$;
\item if $C_2(x_1,x_2)$ holds for $x_1,x_2\in C_1(z)$ not in the same component and there is some $y\in A$ with $\para(x_1,y)$ and $\para(x_2,y)$, then $\para(x_1,x_2)$;
\item if $x\neq y\in A_i\cap A_j$, then $N_k(x,y)$ holds for some $k$;

\item for any finite subset $J_0\subset J$ the $L$-structure $A_{J_0}=\bigcup_{j\in J_0} A_j$ has a removable component, i.e. there is some  $k\in J_0$ such that $A_k\cap A_{J_0\setminus\{k\}}$ is contained in $\cl(b,x_1)\cup\cl(b,x_2)$ for some $b\in A_k$ and $\ort(x_1,x_2)$. 
\end{enumerate}

\end{definition}
\begin{remark}\label{rem:admissible} Suppose $A=\bigcup_{j\in J} A_j$ is admissible.
\begin{enumerate}
\item \label{rem:unique decomp of admissible sets} The components are unique.
\item \label{rem:component extensions} If for $j\in J$ the block $B_j$ contains $A_j$, then also $\bigcup_{j\in j} B_j$ is an admissible $L$-structure in a natural way.

\item \label{rem:intersections of components}
If $A=\bigcup_{j\in J} A_j$ is admissible with $A_j \cong \af_2$ for all $j\in J$, then by by condition (7), we see that for $i\neq j\in J$, the set $X=A_i\cap A_j$ is  either empty, a single vertex, or of the form $X=\cl(\ell)$ for some line $\ell$ or $X=\cl(\ell)\cup\cl(\ell')$ for orthogonal families $\cl(\ell), \cl(\ell')$.

\end{enumerate}

\end{remark}

\begin{theorem}\label{thm:models are admissible}
 An $L$-structure $M$ is a model of $T$ if and only if $M$ is admissible with all components isomorphic to $\af_2$ and for any $x_1,x_2\in M$ with $\ort(x_1,x_2)$ there is a component containing $x_1$ and $x_2$.
\end{theorem}
\begin{proof}
First assume that $M$ is admissibile and all components are isomorphic to $\af_2$ and such that and for any $x_1,x_2\in M$ with $\ort(x_1,x_2)$ there is a component containing $x_1$ and $x_2$. Then it is immediate that $M$ is a model of $T$.

Now assume conversely that $M$ is a model of $T$. Then every $E$-edge $e$ is contained in a unique component, i.e. copy of $\af_2$, by Axiom~\ref{unique_extension}. 
For any two components $A\neq B$ in $M$, we have $A\cap B\subseteq\cl(\ell)$ for some line $\ell$ by Axiom~\ref{block_expander} and  Axiom~\ref{unique_extension}.

Part (b) follows from Proposition~\ref{prop:cycles}: suppose $B_0,\ldots B_n$ form a minimal counterexample and that they are ordered in such a way that we can pick $a_i\in B_{i-1}\cap B_i$. Then we can build a simple cycle connecting these $a_i, i=0,\ldots n,$ contradicting Proposition~\ref{prop:cycles}.
\end{proof}
We note the following for reference below:
\begin{remark}\label{rem:admissible extension}
Suppose $M$ is admissible and let $A$ be a block of level $k\geq 1$.
Assume one of the following:
\begin{itemize}
\item there are $x_1, x\in M$ and $A\cap M\subset \cl(x_1,x)$; or
\item  there are $x_1, x_2\in M$  not contained in a common block with $\ort(x_1,x_2)$ and $x_1, x_2$ and  $A\cap M\subset\cl(x,x_1)\cup\cl(x,x_2)$.
\end{itemize}
Let $M\cup A$ be the $L$-structure defined as follows:

\begin{enumerate}
\item for $x\in M\setminus A, y\in A\setminus M$ we put $\neg C(x,y), \neg E(x,y), \neg B_k(x,y)$, and $\neg N_k(x,y)$ for $k\in\N$;

\item $C_k(x,y), k\geq 2,$ holds if and only if there is a unique shortest $C$-path of length $k$ from $x$ to $y$;
\item 
 if $C_2(x_1,x_2)$ holds for $x_1,x_2\in C_1(z)$ not in the same component and there is some $y\in A$ with $\para(x_1,y)$ and $\para(x_2,y)$, then $\para(x_1,x_2)$;
 
 \item  if $C_2(x_1,x_2)$ holds for $x_1,x_2\in C_1(z)$ not in the same component and there is some $y\in A$ with $\para(x_1,y)$ and $\ort(x_2,y)$, then $\ort(x_1,x_2)$;

\item if $B_k(x,y)$ holds, then $x, y\in M$ or $x, y\in A$.

\end{enumerate}  
Then $M\cup A$ is admissible. 

\end{remark}

\begin{definition}\label{def:class strong ext}
For admissible structures $A=\bigcup_{i\in J} A_i, B=\bigcup_{j\in J'} B_j, J\subseteq J',$ we write $A\leq B$ and say that $A$ is strong (or strongly embedded) in $B$ if  $A\subset B$ and 
\begin{enumerate}
\item  for $j\in J$ we have $A_j=B_j\cap A$;
\item  and for every finite subset $J_0\subset J', J_0\not\subseteq J,$ the admissible set $B_{J_0}= \bigcup_{j\in J_0} B_j$ has a removable component $B_k$ with $k\in J_0\setminus J$.

\end{enumerate}

Let $(\K,\leq)$ be the class of finite admissible structures with the partial ordering $\leq $.
\end{definition}

\begin{remark}\label{rem:admissible sets in af2}
\begin{enumerate}
\item Note that any block $B\subseteq\af_2$ of level $k\geq 1$ in $\af_2$ is an admissible structure and $B\leq\af_2$.

\item \label{rem: components are strong}
Clearly, if $A=\bigcup_J A_j,$  each $A_j$ is strong in  $A$.

\end{enumerate}
\end{remark}

\begin{definition}\label{def:minimal ext}

Let $A=\bigcup_{j\in J} A_j, B=\bigcup_{j\in J'} B_j\in\K$ with $A\leq B$. We call $B$ a minimal strong extension if exactly one of the following holds:
\begin{enumerate}
\item $A$ and $B$ have the same number of components and at most one component of $B$ properly contains the corresponding component of $A$; 
\item $B$ has exactly one component $B_0$ such that $B_0\cap A$ is not a component of $A$ and $A=\bigcup _{j\neq 0}  B_j$.
\end{enumerate}

\end{definition}

\begin{remark}\label{rem:minimal ext}
Clearly, for $A, B$ are admissible structures with finitely many maximal blocks and $A\leq B$, we can find a finite sequence $A=B_0\leq B_1\leq\ldots\leq B_s=B$ where each extension $B_i\leq B_{i+1}, i=1,\ldots s-1$ is minimal.
\end{remark}

The following is crucial for our approach:

\begin{proposition}\label{prop:amalgam}
The class $(\K,\leq)$  of finite admissible structures has amalgamation and joint embedding with respect to $\leq$. 
\end{proposition}

\begin{proof}
Let $A, B, C\in\K$ with $A\leq B, C$. It suffices to consider the case where $C$ is a minimal strong extension of $A$.

First suppose $A, B, C\subset \af_2$ are blocks with $A\subset B, C$ and $A=\Exp_k(e)$ has level $k\geq 1$.  Let $k'$ be minimal such that $B, C\subseteq D=\Exp_{k'}(e)$. Then we call $D$ the canonical amalgam of $B, C$ over $A$. Since $\af_2\subset \dcl(A)$ by Remark~\ref{rem:sticks}, this is well-defined.

Next suppose that $C= A\cup C_0$ where $C_0$ extends the block $A_0\subset A$. Let $B_0\subset$ be the corresponding block in $B$ and let $D_0$ be the canonical amalgam of  $B_0$ and $C_0$ over $A_0$. Then by Remark~\ref{rem:admissible extension} (\ref{rem:component extensions}), the $L$-structure $D=B\cup D_0$ is admissible and a strong extension of $B$ and $C$.

Now suppose that $C=A\cup C'$ where $C'\cap A\subset \cl(x_1, x)\cup\cl(x_2,x)$ for some orthogonal pair $x_1, x_2\in A$. 
Then we distiguish two cases:

First suppose there are $z\in C'\setminus A, z'\in B_j\setminus A$ for some $j\in J$ with $z, z'\in D_1(x_1)\cap D_1(x_2)$. Then we identify $z$ and $z'$ and we let $D_j$ be the smallest block  containing $B_j$ and $C'$.  The $L$-structure $D=B\cup D_j$ is admissible and a strong extension of $B$ and~$C$.

If this is not the case, we are in one of  the situations of Remark~\ref{rem:admissible extension} and again we conclude that $D=B\cup C$ is an admissible strong extension of $B$ and~$C$.
\end{proof}

\begin{theorem}\label{thm:main}
    Let $(\K,\leq)$ be the class of finite admissible structures with strong embeddings.
    A model $M$ of $T$ is $\K$-saturated if and only if $\omega$-saturated.
    Therefore, $T$ is complete and the theory of $\af_2$.
\end{theorem}
\begin{proof}
It suffices to show that if $M$ is an $\omega$-saturated model of $T$, then $M$ is $\K$-saturated. From this we conclude in the standard way (see e.g. \cite{TZmodel theory}, Theorem 10.4.9): if $N$ is a model of $T$, we may assume that $N$ is $\omega$-saturated, hence $\K$-saturated.
 Any two $K$-saturated structures are partially isomorphic and hence elementarily equivalent, see e.g. \cite{TZmodel theory} Exercises  4.4.1 and 1.3.5. This then shows that $T$ is complete.
   
Now suppose that $M$ is an $\omega$-saturated model of $T$. To show that $M$ is $\K$-saturated we have to show the following:
\begin{itemize}
\item every admissible structure with a single component can be strongly embedded into $M$;
\item  for every pair $A, B \in K$ with $A \leq B$, and every strong embedding of $A$ 
    into $M$ we can find a strong copy of $B$ over $A$ in $M$.
\end{itemize}

This then implies also that every admissible structure strongly embeds into $M$.

For the first point, let $A\in\K$ be an admissible structure with a single component, so $A$ is isomorphic as an $L$-structure to some block of level $k\geq 1$.
Since $M$ is a model of $T$, every $E$-edge $e$ in $M$ is contained in a unique extension $\Exp_k(e)$.
Now, we see that $A$ is strongly embedded into $M$ by Proposition~\ref{prop:cycles}.

For the second point assume that we have $A, B\in\K$ with $A\leq B$ and $A$ strongly embedded into $M$. By Remark~\ref{rem:minimal ext} we may assume that $B$ is a minimal strong extension of $A$, so either $B$ extends a unique component of $A$ or contains a new component.

If $B$ extends a unique component $A_0$ of $A$, then by Axiom~\ref{unique_extension} we can expand this component in $M$ and this extension is clearly strongly embedded.

Now assume that $B=A\cup B_0$ for some component $B_0$ of $B$ with $A\leq B$. Then $B_0$ is removable in $B$ and hence $B_0\cap A$ is contained in $\cl(x_1,x)\cup\cl(x,x_2)$ for some orthogonal families. Since $M$ is $\omega$-saturated, there is a copy of $\af_2$ in $M$ intersecting $A$ in the prescribed way and by Proposition~\ref{prop:cycles} this extension is strongly embedded. Note that if $B_0\cap A$ is contained in a set of the form $\cl(x_1,x)$, then $M$ contains infinitely many copies of $B_0$ over (the image of) $A$ with the prescribed intersection. Otherwise, by Axioms~\ref{block_expander} and~\ref{relations hold only in blocks}, $M$ contains a unique block intersecting the image of $A$ in  $\cl(x_1,x)\cup\cl(x,x_2)$. 
\end{proof}
 
The following is an immediate consequence:
\begin{corollary}\label{cor:saturated}
Let $M$ be a saturated model of $\T$. Then every vertex $x$ and every set $\cl(\ell)$ are contained in infinitely many copies of $\af_2$ and these copies are conjugate over $x$ or $\cl(\ell)$, respectively.
\end{corollary}

The proof of Theorem~\ref{thm:main} also shows the following useful consequence:
\begin{corollary}\label{cor:strong submodels}
Let $M$ be a model of $\T$ with a submodel $M_0$. Then $M_0$ is a strong submodel of $M$ if and only if for each connected component $M_1$ of $M$ the model $M_1\cap M_0$ is connected.
In particular, any path connected submodel of $M$ is strongly embedded.
\end{corollary}

\section{Stability}

\begin{theorem}\label{thm:stable}
The theory $\T$ of the Free Factor Complex of rank 2 is $\omega$-stable and $\af_2$ is its prime model.
\end{theorem}

The proof follows from the following lemma:
\begin{lemma}\label{lem:types}
Let $M$ be a saturated model and $M_0$ a countable strongly embedded submodel. Let $c\in M $ and $a_1\in M_0$ be such that $d(c,a_1)$ is minimal. Let $p=(c=t_0,\ldots,t_l=a_1,t_{l+1}=a_2)$ be a path with block cover $(B_0,\ldots, B_n)$ for  some $a_2\in M_0$ where $B_{i-1}\cap B_i=\cl(\ell_i)$ for some line $\ell_i, i=1,\ldots, n$.

Let $C_i\subset B_i,i=0,\ldots n$ be blocks of level $k_i$ such that $p\subset \bigcup C_i$ and let $z_i\in p\cap B_{i-1}\cap B_i, i=1,\ldots n,$ and suppose $D_{n_i}(z_i,z_{i+1})$ holds.
Let $q(x)\subset\tp(c/A)$ be the partial type consisting of the following formulas:
 \begin{itemize}
 \item there exist $x_1,\ldots,x_{l-1}$ forming a path from $c$ to $a_1$ and 
 $z_1,\ldots, z_n\in\{x_1,\ldots, x_{l-1}\}$ with $D_{n_i}(z_i,z_{i+1})$, $D_{n_0}(x,z_1)$ 
 and $D_{n_n}(z_n,a)$;
 \item  $B_{k_i}(z_i,z_{i+1}), i=1,\ldots n,$ and $B_{k_0}(x,z_1)$ and $B_{k_n}(z_n, a_2)$;
 \item the quantifier-free formula describing the part of $p$ of contained in $C_i$ in 
 the diagram  in the block of level $k_i$;
 \item $\neg B_n(x_i,x_{i+2})$ for all $n\in\N$;
 \item for any pair $x_i, x_j$ with $C_2(x_i,x_j)$ specify $\ort(x_i,x_j)$ or $\para(x_i,x_j)$.
 
 \end{itemize}
If $d\in M$ realizes $q(x)$, then there is an automorphism of $M$ fixing $M_0$ and taking $c$ to~$d$.

\end{lemma}
\begin{proof}
If $d$ realizes $q(x)$, then there is a path $p'$ from $d$ to $a_1$ with block cover $(D_0,\ldots, D_n)$ isomorphic to $(B_0,\ldots, B_n)$. The map taking $p$ to $p'$ and fixing $M_0$ extends to an isomorphism from $M_0\cup\bigcup_{i=0}^n B_i$ to $M_0\cup\bigcup_{i=0}^n D_i$.
\end{proof}

\begin{proof}[Proof of Theorem~\ref{thm:stable}]
Over a countable $M_0$ there are clearly only countably many possibilities for $a_1$ and the types of paths $p$. Hence $\T$ is $\omega$-stable. Since $\Aut(\F_2)$ act transitively on edges of $\af_2$  and $\af_2$ is algebraic in any edge, it follows that $\af_2$ realizes exactly the isolated types and hence it is the prime model of $\T$.

\end{proof}

\section{Algebraic closure}\label{sec:acl}

We describe the algebraic closure only for subsets of copies of $\af_2$. The general case follows 
fairly easily from this using Theorem~\ref{thm:main}. We omit the details for the sake of brevity.
\begin{lemma}\label{lem:acl}
Let $M \models Th(\af_2)$ and let $A\cong \af_2$ be a component of $M$.
\begin{enumerate}
\item for any $X\subset A$ we have $\acl(X)\subseteq A$;
\item\label{lmm:non-conj-copy-acl-af2} if $x,y\in A$ and $x\notin C(y)$, then $\acl(x,y)=A$; in particular, for any
$E$-edge $e=(x,y)$  we have $\acl(e)=A$;
\item if $x,y\in A$ with $\ort(x,y)$, then $\acl(x,y)=A$;
\item if $x,y\in A$ with $C(x,y)$, then $y\notin\acl(x)$ and $z\notin\acl(x,y)$ for $z\in\ell(x,y)$;

\item if $X\subset\cl(\ell), |X|\geq 2,$ for some line $\ell$ in $M$, then $\acl(X)=\cl(\ell)$.

\item if $x\in C(b)$ and $x,y\not\subset\cl(\ell)$ for any line $\ell$, then $\acl(x,y)=A$.

\item if $x\in M$ is a vertex, then $\acl(x)=\{x\}$.

\end{enumerate} 

\end{lemma}

\begin{proof}
(1): This follows from Corollary~\ref{cor:saturated}.

(\ref{lmm:non-conj-copy-acl-af2}): First assume that $x\in D_1(y)$. Then  by Fact~\ref{fct:edge_acl}
   we have $\acl(x,y)=A$.  

Now consider the general case and let $B$ be a block of level $k$ containing $x$ and $y$ and let $y_2\in B\cap D_1(x)$.
By Proposition~\ref{prop: less trivial bounds on blocks}.\ref{lmm:not-conj} we see that $\Exp_{k'}(x,y_2)$   contains all blocks of level $k$ containing $x$ and $y$  where $k'$ is as computed in that proof.
    Therefore, there are only finitely many such blocks, and hence $B\subset\acl(x,y)$.  Since $B$ contains an edge, it follows from the first part that $\acl(x,y)=A$.

(3):  If $\ort(x,y)$, then $z\in C_1(x)\cap C_1(y)$ is algebraic in $x,y$, and so $z_0\in\ell(z,x)\cap\ell(z,y)$ is algebraic in $x,y$ and not conjugate to $x$. Hence the statement follows from (2)
    
(4)   This is a direct consequence of Corollary~\ref{cor:C1 trans}.
     
(5)  By Remark~\ref{rem:useful things about nbh closed} we have $\cl(\ell)\subseteq\acl(X)$. Now suppose towards a contradiction that there is some $z\in\acl(X)\setminus\cl(\ell)$ and $\cl(\ell)\subseteq C(x)$. Then by (2) and (3) we must have $z\in C(x)$ and the unique $C$-path  $(z=z_0,\ldots, z_k=x)$ from $z$ to $x$ is contained in $\acl(X)$. Since $z\notin\cl(\ell)$, there is some  pair $z_{i-1}, z_{i+1}$ with $\ort(z_{i-1},z_{i+1})$ and thus by part (3) we have $\acl(X)=A$, contradicting Remark~\ref{rem: fix nbhs and act on lines}.

(6) If $x\in C(y)$ and $x,y\not\subset\cl(\ell)$ for any $\ell$, the unique $C$-path $(x=z_0,\ldots, z_n=y)$ from $x$ to $y$ contains $z_{i-1}, z_{i+1}$ with $\ort(z_{i-1}, z_{i+1})$ by Lemma~\ref{lem:conjugates} and hence $\acl(x,y)=A$ by part (3).

(7) Let $x\in A$ and suppose $y\in\acl(x)$. Then by (2) and (3) we must have $z\in\C(x)$ and hence the unique $C$-path from $x$ to $z$ is in $\acl(x)$, contradicting again Part (3).

\end{proof}
\begin{remark}\label{rem:ack}
Note that an algebraically closed subset $A\subset\af_2$ is either a single vertex, of the form $\cl(\ell)$ for some line $\ell$ or all of $\af_2$. In particular, $A=\acl(x,y)$ for at most two elements $x,y\in\af_2$.

As in the case of right-angled buildings (see \cite{TENT_2014}), it can be shown  that for an arbitrary model $M$ of $\T$ and $X\subset M$ we have
\[\acl(X)=\bigcup_{x,y\in X}\acl(x,y).\]
\end{remark}

\section{From models of $Th(\af_2)$ to models of $Th(\of_2)$}

We end this note with a comment on the relation between models of $\T$ and models of the Theory of the Farey graph.
\begin{definition}
Given a model $M$ of $Th(\af_2)$ we define
a graph $M^c$ as follows:  
\begin{itemize}
    \item The vertices of $M^c$ are the $C$-connected components of vertices in $M$.  
    \item Two vertices  are adjacent in $M^c$ if there exist representatives from each class that are adjacent in $M$ with respect to $E$.  
\end{itemize}

There is a natural map
\[
\pi : M \longrightarrow M^c, \quad x \mapsto [x],
\]
sending each element of $M$ to its $C$-connected component.  

We call $M^c$ the \emph{conjugacy quotient} of $M$.
\end{definition}

We note that the $C$-connected components of a vertex in models of $\Th(\af_2)$  are $\bigvee$-definable.

\begin{remark}\label{rem:of2}
    By definition of $\of_2$ as the quotient of $\af_2$ under conjugation, it is immediate that if $M$ is a model of $\af_2$, then $M^c$ is a union of copies of $\of_2$ and these copies can only intersect in a single vertex by Theorem~\ref{thm:models are admissible}.
\end{remark}

\begin{proposition}\label{prop:model-af2-of2}
    For every $M \models Th(\af_2)$, $M^c$ is a model of the theory of $\of_2$.
\end{proposition}
\begin{proof}
Given $M^c$ consider the associated graph $\Gamma$ as in \cite{MT}: the  vertices of $\Gamma$ correspond to the copies of $\of_2$ in $M^c$ and there is an edge between two vertices if the corresponding copies of $\of_2$ have a common vertex. 
By \cite{MT}, Theorem 2.23 it suffices to show that $\Gamma$ has no cycles and this follows directly
from Proposition~\ref{prop:cycles}.
\end{proof}

\section{Appendix}\label{sec:appendix}
In this appendix we describe the quantifier-free diagram of the blocks $\Exp_k(e)$ where $e=E(x,y)$.

Let  $X_1=\Exp_1(x,y)$. Then the four common neighbours of $x,y$, which we call sticks following \cite{BB}, are the vertices $xy, x^{-1}y, xy^{-1}, x^{-1}y^{-1}  $ and
the following additional relations hold hold:
\begin{itemize}
    \item $C( xy,x^{-1}y^{-1})$;
    \item $C(x^{-1}y, xy^{-1})$.
\end{itemize}

Then $X_2=\Exp_2(x,y)$ has the following additional vertices corresponding to the eight edges of $\Exp_1(x,y)\setminus\{(x,y)\}$. The new vertices of $X_2$ are of three different types.
For each of the eight edges in $X_1\setminus\{(x,y)\}$ there is a unique new vertex with the following additional relations:
   
    \begin{enumerate}
            \item $x^{2}y$ is a stick of $\{x, xy\}$,   has an $E$-edge with $y^{-1} x y$ and a $C$-edge with $xyx$;
        \item $x^{-2}y$ is a stick of $\{x, x^{-1}y\}$,   has an $E$-edge with $y^{-1} x y$ and a $C$-edge with  $xy^{-1}x$;
        \item $x^{2}y^{-1}$ is a stick of $\{x, xy^{-1}\}$,   has an $E$-edge with $y x y^{-1}$  and a $C$-edge with  $xy^{-1}x$;
        \item $x^{-2}y^{-1}$ is a stick of $\{x, x^{-1}y^{-1}\}$  has an $E$-edge with $y^{-1} x y$  and a $C$-edge with  $xyx$;
        
        \item $x\inv y^{-2}$ is a stick of $\{y, x\inv y\inv\}$, has an $E$-edge with $yxy\inv$  and a $C$-edge with  $yxy$;
        \item $x\inv y^2$ is a stick of $\{y, x^{-1}y\}$, has an $E$-edge with $x\in y^{-1} x$  and a $C$-edge with  $yx\inv y$;
        \item $xy^{-2}$ is a stick of $\{y, xy\inv \}$, has an $E$-edge with $xyx\inv$ and a $C$-edge with  $yx\inv y$;
        \item $xy^2$ is a stick of $\{y, xy\}$, has an $E$-edge with $x\inv y x$  and a $C$-edge with  $yxy$.
        
        \end{enumerate}

        The next set of vertices consists of verticse which are  common extensions for two edges sharing a vertex, the first four are $1$-conjugates of $x,y$, respectively, the last four have no $C$-edges.
        
        \begin{enumerate}

        \item $xyx^{-1}$ is a stick for $\{x, xy\}, \{x, xy^{-1}\}$ and has a $C$-edge with  $y$;
        \item $x^{-1}yx$ is a stick for $\{x, x^{-1}y\}, \{x, x^{-1}y^{-1}\}$  and has  a $C$-edge with  $y$;   
        \item $yxy^{-1}$ is a stick for $\{y, x^{-1}y^{-1}\}, \{y, xy^{-1}\}$ and has  a $C$-edge with  $x$ ;
        \item $y^{-1}xy$ is a stick for $\{y, x^{-1}y\}, \{y, xy\}$ and  has a $C$-edge with  $x$;

        \item $xyx$ is a stick of $\{x, xy\}$ and $\{x, x^{-1}y^{-1} \}$.
        \item $xy^{-1}x$ is a stick of $\{x, xy^{-1}\}$ and $\{x, x^{-1}y \}$.
        \item $yxy$ is a stick of $\{y, x^{-1}y^{-1}\}$ and $\{y, xy\}$.
        \item $yx^{-1}y$ is a stick of $\{y, x^{-1}y\}$ and $\{y, xy^{-1} \}$.

\end{enumerate}

To completely describe the $L$-structure on $X_2$ we also note that
the following pairs of vertices are orthogonal  with witness $x$:

\begin{itemize}
\item $xyx\inv$ and $x\inv y x$;
\item $x^{2}y^{-1} $ and $x^{-2}y$;
\item  $x^{2}y$ and $ x^{-2}y^{-1}$;
\end{itemize}

and the following pairs are orthogonal with witness $y$:
\begin{itemize}
\item $yxy\inv$ and $y\inv xy$;
\item $y^{2}x^{-1}$ and $ y^{-2}x$; and
\item $y^{2}x$ and $y^{-2}x^{-1} $.
\end{itemize}

All other pairs of vertices $u,v\in X_2$ with $u\in C_2(v)$ are parallel.

Now $X_{k+1}=\Exp_{k+1}(x,y)$ is constructed from $X_k$ as follows:

Take every edge $e$ in $X_k$ with $\Exp_1(e)\subset X_k$ but $2-Ext(e) \not\subset X_k$ and add all sticks of edges in $\Exp_1(e)$ which are not already in $X_k$ as described above.

\begin{remark}   Note that the new $E$-edges never reduce the distance between vertices and if a new vertex has  at least three $E$-edges, then two of the neighbours form an edge. Furthermore, every new vertex has at most one $C$-edge. 

\end{remark}

\begin{corollary}\label{cor:blocks are convex}
The $x,y\in X_k=\Exp_k(e)$ the following holds:
\begin{enumerate}
\item If $x\in C(y)$, then the unique $C$-path from $x$ to $y$ belongs to $X_k$.
\item If $x,y\in X_{k-1}$, then the $E$-distance of $x$ and $y$ in $X_k$ does not decrease.
\item If $x,y$ are orthogonal, then the witness for orthogonality is in $X_k$.
\end{enumerate}

\end{corollary}

\begin{proof}
For (1) assume that there are $x,y\in X_k$ such that the $C$-path from $x$ to $y$ is in $X_{k+1}$ and has length $d$ with $d$ minimal. Then this path contains $d-1$ new vertices and $d$ new $C$-edges, which is impossible because every new vertex has at most one $C$-edge.

For (2) we already noted that new $E$-edges never reduce distance and for (3) we noted that any new vertex of valency at least $3$ has two neighbours forming an edge. Thus, they are not orthogonal, so a new vertex is not a witness to orthogonality of previous vertices.
\end{proof}

\end{document}